\documentclass[12pt]{article}

\usepackage{graphicx}
\usepackage{amssymb}
\usepackage{epstopdf}
\usepackage{verbatim}
\usepackage{amsthm,  mathrsfs, enumerate, cite}
  
\usepackage{amsmath}
\usepackage{MnSymbol}

\newtheorem{theorem}{Theorem}[section]
\newtheorem{lemma}[theorem]{Lemma}

\newtheorem{corollary}[theorem]{Corollary}

\numberwithin{equation}{section}

\newcommand{\req}[1]{Eq.\,(\ref{#1})}

\begin{document}

\title{A Recursive Method for Computing Certain Bessel Function Integrals }

\author{Jeremiah Birrell$^{a}$\\
\vspace{-.1in}
\scriptsize{$^{a}$Department of Physics and Program in Applied Mathematics, The University of Arizona, Tucson, Arizona, 85721, USA} }

\maketitle

\begin{abstract}We investigate a family of integrals involving modified Bessel functions that arise in the context of neutrino scattering. Recursive formulas are derived for evaluating these integrals and their asymptotic expansions are computed.  We prove in certain cases that the asymptotic expansion yields the exact result after a finite number of terms.  In each of these cases we derive a formula that bounds the order at which the expansion terminates. The method of calculation developed in this paper is applicable to similar families of integrals that involve Bessel or modified Bessel functions.
\end{abstract}

\section{Introduction}

In this paper we study linear combinations of the family of integrals
\begin{equation}\label{F_def}
F_\nu^N(z)=\int_0^\infty \cosh^N(t) K_\nu(z\cosh(t)) dt,\hspace{2mm}  \Re(z)>0,
\end{equation}
where the modified Bessel function of the second kind for $\Re(z)>0$ and $\nu\in\mathbb{C}$ is given by the integral
\begin{equation}\label{BesselK_def}
K_\nu(z)=\int_0^\infty\exp(-z\cosh(t))\cosh(\nu t)dt.
\end{equation}
See, for example, \cite{Watson} for a detailed treatment of the theory of Bessel functions. 

The following special case is of particular interest
\begin{equation}\label{G_def}
G_\nu^N(z)=\int_0^\infty \cosh^N(t) K_\nu(z\cosh(t))\sinh^2(t) dt,\hspace{2mm}  \Re(z)>0.
\end{equation}
Note that $G_\nu^N$ can be obtained from the $F^N_\nu$ by the linear combination
\begin{equation}
G_\nu^N(z)=F^{N+2}_{\nu}(z)-F^N_\nu(z).
\end{equation}
Integrals of the form \req{F_def} appear in the computation of the reaction rate for electron-positron annihilation into neutrinos, as outlined in Appendix \ref{app:neutrino}.

The integral \req{F_def} is related to the Meijer G-function  function and  Mathematica~\cite{Mathematica} is capable of expressing the general case \req{F_def} in terms of generalized hypergeometric functions. In addition, a related integral formula is well known, see equation $13\!\cdot\! 72$ (1) in \cite{Watson} or equation 2.16.13.2 in \cite{Prudnikov}.
\begin{lemma}\label{Watson_int2}
Let  $\Re(z)>0$ and $a,b\in\mathbb{C}$.  Then
\begin{equation}\label{K_prod_formula}
2\int_0^\infty K_{a+b}(2z\cosh(t))\cosh((a-b)t)dt=K_a(z)K_b(z).
\end{equation}
\end{lemma}

By expanding 
\begin{equation}\label{cosh_exp}
\cosh^N(t)=\sum_{k=0}^N  c_k\cosh(kt),
\end{equation}
\req{K_prod_formula} can be used to express $F_\nu^N(z)$ as a linear combination of products of modified Bessel functions.

 In this paper an alternative, and in our opinion simpler, method for computing \req{F_def} will be given  that also provides more detailed information about the structural form of the result in some cases.  In particular, we show how to compute these integrals recursively, using only the case $a=b$ from \req{K_prod_formula} and without explicitly performing the expansion \req{cosh_exp}.  We will also prove that if  $N$ is even and $\nu$ is odd or vice versa then $F^N_\nu(z)$ can be expressed as a finite linear combination of terms of the form $e^{-z}z^{-n}$, and compute the order $n$ at which these expansions terminate.  In general, the order is smaller than one would naively expect from combining the expansion \req{cosh_exp} with \req{K_prod_formula}.  More specifically,  the following theorem will be proven.
\begin{theorem}\label{main_theorem}
Let $N,\nu\in\{0,1,2...\}$. If $N$ is even and $\nu$ is odd or vice versa then
\begin{equation}\label{F_formula}
F_\nu^N(z)=z^{-1}e^{-z}\sum_{s=0}^{M_{N,\nu}}\left[\sum_{j+r=s}a_j\Gamma(r+1/2)\sum_{k+l=r}b_{k,N} c_{l,j}\right] z^{-s}
\end{equation}
where 
\begin{align}
M_{N,\nu}=&\begin{cases} \nu-1 &\mbox{if } N<\nu, \\ 
 N-1& \mbox{if } N>\nu \end{cases}
\end{align}
and
\begin{equation}\label{G_formula}
G_\nu^N(z)=z^{-2}e^{-z}\sum_{s=0}^{K_{N,\nu}}\left[\sum_{j+r=s}\!\!a_j\Gamma(r+3/2)\!\!\sum_{k+l=r}\!\!\!d_{k,N} c_{l,j}\right] z^{-s}
\end{equation}
where $K_{N,\nu}=M_{N+2,\nu}-1$.  

See \req{a_def} through \req{c_def} for the definition of the coefficients $a_j$, $b_{j,N}$, $d_{j,N}$, and $c_{l,j}$. 
\end{theorem}

Section \ref{sec:asymptotics} contains  necessary background material on asymptotic expansions, including explicit formulas for the asymptotic expansions of $F_\nu^N(z)$ and $G_\nu^N(z)$ as $z\rightarrow\infty$. Section \ref{sec:recursion} contains proofs of  several recursion relations that the family of integrals \req{F_def} obeys. Using these recursion relations, a proof of  Theorem \ref{main_theorem} will be given in section \ref{sec:poly_degree}. The case of interest to neutrino scattering is $N$ and $\nu$ both odd.  Although Theorem \ref{main_theorem} does not apply, the recursion relations derived in section \ref{sec:recursion} still allows for an  analysis of this case. This is discussed in Appendix \ref{app:neutrino}.


\section{Asymptotics}\label{sec:asymptotics} 

In this section we collect some needed results on asymptotic expansions of integrals with exponentially decaying integrands.  For these purposes, it is convenient to rewrite the integrals \req{F_def} and  \req{G_def} as   
\begin{align}\label{F_def2}
F_\nu^N(z)=&\int_0^\infty \frac{(x+1)^N}{\sqrt{x(x+2})} K_\nu(z(x+1)) dx,\\
G_\nu^N(z)=&\int_0^\infty \sqrt{x(x+2)}(x+1)^N K_\nu(z(x+1)) dx.
\end{align}
We will consider these integrals for $z,N,\nu\in\mathbb{R}$ as $z\rightarrow\infty$.

Laplace's method will be needed. See, for example, \cite{Olver}.
\begin{lemma}\label{laplace_integrals}
Let $q:(0,\infty)\rightarrow\mathbb{R}$ and $b>0$ such that $q\in L^1(e^{-zx}dx)$ for all $z\geq b$.  If
\begin{equation}
q\sim\sum_{n=0}^\infty a_nx^{(n+\lambda-\mu)/\mu} \text{ as } x\rightarrow 0
\end{equation}
for $\lambda,\mu>0$ then
\begin{equation}
\int_0^\infty e^{-zx}q(x)dx\sim \sum_{n=0}^\infty \Gamma\left(\frac{n+\lambda}{\mu}\right)\frac{a_n}{z^{(n+\lambda)/\mu}} \text{ as } z\rightarrow\infty.
\end{equation}
\end{lemma}

In fact, we will also need a generalization of Laplace's method to integrals where the kernel $e^{-x}$ is replaced by another exponentially decaying function.  The general method for deriving an asymptotic expansion in this scenario is given by Olver in Section 9.6  of \cite{Olver}.  For the present purposes, the following variant is needed.  We include a proof for completeness.
\begin{lemma}
Let  $K:(0,\infty)\rightarrow \mathbb{R}$ have an asymptotic expansion
\begin{equation}
K(x)\sim x^\alpha e^{-\beta x}\sum_{n=0}^\infty \frac{a_n}{x^n} \text{ as } x\rightarrow\infty
\end{equation}
where $\alpha\in\mathbb{R}$, $\beta>0$ and $f\in O(e^{\sigma x})$ as $x\rightarrow\infty$ for some $\sigma>0$. If $K$ is bounded on compact subsets of $(0,\infty)$ (but not necessarily around $0$) and $f$ is integrable on bounded sets then for any $M>0$, and sufficiently large $z$
\begin{align}\label{K_exp}
&\int_0^\infty f(x)K(z(x+1))dx\\
=&\sum_{n=0}^{M-1}a_nz^{\alpha-n}e^{-\beta z}\int_0^\infty f(x)(x+1)^{\alpha-n}e^{-\beta zx}dx+O(z^{\alpha-M}e^{-\beta z}).\notag
\end{align}
If $f$  is also bounded on bounded sets then the error term can be improved to $O(z^{\alpha-M-1}e^{-\beta z})$.
\end{lemma}
\begin{proof}
 In this proof  we employ the notational device of using the same symbol for different constants used in bounding asymptotic behavior, here denoted by $C$.

First it must be shown that for sufficiently large $z$, both sides are defined.  For the left hand side, fix $a>0$ large enough so that $|f|<Ce^{\sigma x}$ and $|K|<Cx^\alpha e^{-\beta x}$ on $(a,\infty)$ and take $z>a,\sigma/\beta$.  Then 
\begin{align}
&\int_0^\infty| f(x)K(z(x+1))|dx\\
\leq& \int_0^a|f(x) K(z(x+1))|dx+C\int_a^\infty e^{-(\beta z-\sigma)x-\beta z}(z(x+1))^\alpha dx <\infty.\notag
\end{align}
The right hand side is similar.  Therefore all integrals are defined for $z$ sufficiently large.

Now compute
\begin{align}
&\left|\int_0^\infty f(x)K(z(x+1))dx-\sum_{n=0}^{N-1}a_nz^{\alpha-n}e^{-\beta z}\int_0^\infty f(x) (x+1)^{\alpha-n}e^{- \beta zx}dx\right|\notag\\
\leq& \int_0^\infty|f(x)\Delta^N(z(x+1))|dx, \hspace{2mm} \Delta^N(x)=K(x)-x^\alpha e^{-\beta x}\sum_{n=0}^{N-1} \frac{a_n}{x^n}.
\end{align}
Take $a$ sufficiently large so that $|\Delta^N(x)|<Cx^{\alpha-N}e^{-\beta x}$ and $|f(x)|<Ce^{\sigma x}$ on $(a,\infty)$ and let $z>a,\sigma/\beta$.  Then
\begin{align}
&\int_0^\infty|f(x)\Delta^N(z(x+1))|dx\\
\leq &C z^{\alpha-N}e^{-\beta z}\left(\int_0^a |f(x)|(x+1)^{\alpha-N} e^{-\beta zx}dx+\int_a^\infty (x+1)^{\alpha-N} e^{-(\beta z-\sigma)x}dx\right)\notag\\
\leq& C z^{\alpha-N}e^{-\beta z}.\notag
\end{align}
If $f$ is also bounded on bounded sets then for any $\epsilon>0$
\begin{align}
&\int_0^\infty|f(x)\Delta^N(z(x+1))|dx\notag\\
\leq& C z^{\alpha-N}e^{-\beta z}\left(\int_0^a  e^{-\beta zx}dx+\int_a^\infty (x+1)^{\alpha-N} e^{-(\beta z-\sigma)x}dx\right)\notag\\
\leq& C z^{\alpha-N}e^{-\beta z}\left(\frac{1}{z}+\int_a^\infty  e^{-(\beta z-\sigma-\epsilon)x}dx\right)\\
\leq& C z^{\alpha-N}e^{-\beta z}\left(\frac{1}{z}+\frac{1}{\beta z-\sigma-\epsilon}\right)\notag\\
\leq&  C z^{\alpha-N-1}e^{-\beta z}\notag
\end{align}
for $z$ sufficiently large.
\end{proof}

Finally, we need to be able to combine asymptotic expansions.
\begin{lemma}
Let $f,g:(0,\infty)\rightarrow\mathbb{R}$ such that
\begin{equation}
f\sim\sum_{n=0}^\infty a_nx^{(n+\lambda-\mu)/\mu}, \hspace{2mm} g\sim\sum_{n=0}^\infty b_nx^{(n+\beta-\mu)/\mu} \text{ as } x\rightarrow 0
\end{equation}
for $\mu,\lambda,\beta>0$. Then 
\begin{equation}
fg\sim \sum_{r=0}^\infty \left(\sum_{m+n=r} a_nb_m\right)x^{(r+\lambda+\beta-2\mu)/\mu} \text{ as } x\rightarrow0.
\end{equation}
\end{lemma}

Together, the above lemmas give the following theorem.
\begin{theorem}\label{asymp_approx}
Let  $K:(0,\infty)\rightarrow \mathbb{R}$  be bounded on compact subsets of $(0,\infty)$ and have the asymptotic expansion
\begin{equation}
K\sim x^\alpha e^{-\beta x}\sum_{n=0}^\infty \frac{a_n}{x^n} \text{ as } x\rightarrow\infty
\end{equation}
where $\alpha\in\mathbb{R}$, $\beta>0$. Let $f\in O(e^{\sigma x})$ as $x\rightarrow\infty$ for some $\sigma>0$, be integrable on bounded sets, and have the asymptotic expansion
\begin{equation}
f\sim\sum_{n=0}^\infty b_nx^{n+\lambda-1} \text{ as } x\rightarrow 0
\end{equation}
for $\lambda>0$.  Then for any $M>0$ and $z$ sufficiently large,
\begin{align}\label{asymp_formula}
&\int_0^\infty f(x)K(z(x+1))dx\\
=&\sum_{s=0}^{M-1}\left[\sum_{j+r=s}\frac{a_j\Gamma(r+\lambda)}{\beta^{r+\lambda}}\sum_{k+l=r}b_k c_{l,j,\alpha} \right] z^{\alpha-\lambda-s}e^{-\beta z}+O(z^{\alpha-\lambda-M}e^{-\beta z})\notag
\end{align}
where
\begin{equation}
(x+1)^{\alpha-j}\sim \sum_{l=0}^\infty c_{l,j,\alpha}x^l, \hspace{2mm} c_{l,j,\alpha}=\frac{(\alpha-j)...(\alpha-j-(l-1))}{l!} \text{ as } x\rightarrow 0.
\end{equation}
\end{theorem}

Applying Theorem \ref{asymp_approx} to $F_\nu^N$ and $G_\nu^N$ yields their asymptotic expansions.
\begin{corollary}\label{FG_asymp}
For $N,\nu\in\mathbb{R}$ the following asymptotic expansions hold as $z\rightarrow\infty$.
\begin{equation}\label{F_asymp}
F_\nu^N(z)=z^{-1}e^{-z}\sum_{s=0}^{M-1}\left[\sum_{j+r=s}a_j\Gamma(r+1/2)\sum_{k+l=r}b_{k,N} c_{l,j}\right] z^{-s}+O(z^{-1-M}e^{-z}),
\end{equation}
\begin{equation}\label{G_asymp}
G_\nu^N(z)=z^{-2}e^{-z}\sum_{s=0}^{M-1}\left[\sum_{j+r=s}a_j\Gamma(r+3/2)\sum_{k+l=r}d_{k,N} c_{l,j}\right] z^{-s}+O(z^{-2-M}e^{-z}),
\end{equation}
where $a_j$ are coefficients in the asymptotic expansion of $K_\nu$ \cite{Watson},
\begin{equation}\label{a_def}
K_\nu(x)=x^{-1/2}e^{-x}\bigg[\sum_{j=0}^{k-1}a_jx^{-j}+O(x^{-k})\bigg],\hspace{2mm} a_j=\sqrt{\frac{\pi}{2}}\frac{\prod_{l=1}^j(4\nu^2-(2l-1)^2)}{8^jj!},
\end{equation}
$b_{j,N}=g_N^{(j)}(0)/j!$ are the coefficients in the Taylor series about $x=0$ of 
\begin{equation}
g_N(x)=\frac{(x+1)^N}{\sqrt{x+2} },
\end{equation}
$d_{j,N}=h_N^{(j)}(0)/j!$ are the coefficients in the Taylor series about $x=0$ of 
\begin{equation}
h_N(x)=\sqrt{x+2}(x+1)^N,
\end{equation}
and
\begin{equation}\label{c_def}
 c_{l,j}\equiv c_{l,j,-1/2}=\frac{(-1/2-j)...(-1/2-j-(l-1))}{m!}
\end{equation}
are the coefficients in the Taylor series of $(x+1)^{-1/2-j}$ about $x=0$.
\end{corollary}

\section{Recursion Relations}\label{sec:recursion}
$F_\nu^N(z)$ is analytic on $\Re(z)>0$  and so it suffices to prove  Theorem \ref{main_theorem} for $z\in\mathbb{R}^+$. $z$ will be restricted to this domain from here on.   The decay of the integrand, uniform in $z$, allows us to differentiate $F_\nu^N(z)$ by differentiating under the integral.  This will yield a pair of recursive formulas for $F_\nu^N(z)$.

The derivations of the recursion relations do not depend on many of the details of the family of integrals in question, so we temporarily generalize to the family of integrals
\begin{equation}\label{H_def}
H_\nu^N(z)=\int_{t_i}^{t_f} r(t)^N C_\nu(zr(t)) w(t)dt
\end{equation}
where the only necessary assumptions are that the integrals are defined in the $L^1$ sense, differentiation commutes with the integral, and $C(x)$ is a solution of the Bessel (upper sign) or modified Bessel (lower sign) equations
\begin{equation}\label{bessel_eq}
z^2\frac{d^2C_\nu(z)}{dz^2}+z\frac{dC_\nu(z)}{dz}=(\nu^2\mp z^2)C_\nu(z)
\end{equation}
and satisfies the recursion relation
\begin{align}\label{C_increment}
\frac{d}{dz}[z^{-\nu}C_\nu(z)]=-z^{-\nu}C_{\nu+1}(z).
\end{align}
These conditions are satisfied by $F^N_\nu$ and $G^N_\nu$ with $C_\nu=K_\nu$ and taking the lower sign.

Using \req{bessel_eq} and differentiating under the integral gives  a recursion relation for $H_\nu^N$.
\begin{lemma}
\begin{equation}\label{F_recursion}
H^{N+2}_\nu(z)=\mp\left(\frac{d^2H_\nu^N(z)}{dz^2}+\frac{1}{z}\frac{d H_\nu^N(z)}{dz}-\frac{\nu^2}{z^2}H_\nu^N(z)\right).
\end{equation}
\end{lemma}
This reduces the task of computing $F_\nu^N$ for non-negative integer $N$ to the base cases $N=0$ and $N=1$, after which the recursion can be used to compute the result for higher $N$.

Similarly, using \req{C_increment} and differentiating under the integral results in a second recursion relation.
\begin{lemma}
\begin{equation}\label{F_recursion2}
H^{N+1}_{\nu+1}(z)=\frac{\nu}{z}H^N_\nu(z)-\frac{d}{dz} H^N_\nu(z).
\end{equation}
\end{lemma}

Using \req{K_prod_formula} the base cases for \req{F_def} can be evaluated.
\begin{lemma}
\begin{align}\label{base_cases}
F_\nu^0(z)=\frac{1}{2}K_{\nu/2}(z/2)^2,\hspace{2mm} F_\nu^1(z)=\frac{1}{2}K_{\frac{\nu+1}{2}}(z/2)K_{\frac{\nu-1}{2}}(z/2).
\end{align}
\end{lemma}
Of particular interest are the cases $N=0$ and $\nu$ odd or $N=1$ and $\nu$ even.  In these cases the results involve modified Bessel functions of half odd integer order, which reduce to finite linear combinations of the terms $e^{-z}z^{-1/2-n}$.  In particular, we will need the following result.
\begin{lemma}\label{K_odd}
$K_{\pm 1/2}(z)=\left(\frac{\pi}{2}\right)^{1/2}z^{-1/2}e^{-z}$ and for $n=1,2,...$, 
\begin{equation}\label{K_leading_order}
K_{\pm(n+1/2)}(z)=\left(\frac{\pi}{2}\right)^{1/2}z^{-1/2}e^{-z}\left(\prod_{j=1}^{n-1}(2(n-j)+1)(z^{-n}+z^{-(n-1)})+R_{n-2}(1/z)\right)
\end{equation}
 where $R_{n-2}$ is a polynomial of degree at most $n-2$.
\end{lemma}
\begin{proof}
This is a consequence of 
\begin{equation}
K_{\pm(n+1/2)}(z)=\sqrt{\frac{\pi}{2z}}e^{-z}\sum_{k=0}^n \frac{(n+k)!z^{-k}}{2^kk!(n-k)!}.
\end{equation}
where  $n$ is a non-negative integer \cite{Watson}.
\end{proof}

Combining \req{K_leading_order} with \req{base_cases} implies
\begin{lemma}\label{F_leading_order}
$F^0_1(z)=\frac{\pi}{2}z^{-1}e^{-z}$ and for $n>0$,
\begin{align}
F^0_{2n+1}(z)=&\frac{2^{2n}\pi\prod_{j=1}^{n-1}(2(n-j)+1)^2}{2}z^{-1}e^{-z}\left(z^{-2n}+z^{-(2n-1)}+R_{2n-2}(1/z)\right),
\end{align}
\begin{equation}\label{F10_F12}
F^1_0(z)=\frac{\pi}{2}z^{-1}e^{-z},\hspace{2mm} F^1_2(z)=\frac{\pi}{2}z^{-1}e^{-z}(1+2z^{-1}),
\end{equation}
and for $n>1$,
\begin{align}
F^1_{2n}(z)=&\frac{2^{2n-1}\pi\prod_{j=1}^{n-1}(2(n-j)+1)^2}{2(2n-1)}z^{-1}e^{-z}\left(z^{-(2n-1)}+z^{-(2n-2)}+R_{2n-3}(1/z)\right).
\end{align}
In the above formulas, $R_{2n-2}$ and $R_{2n-3}$ denote polynomials of degrees at most $2n-2$ and $2n-3$ respectively.
\end{lemma}

 Using Lemma \ref{F_leading_order} and inducting via the recursion relation \req{F_recursion} yields the following.
\begin{lemma}
For $N$ odd and $\nu$ even or vice versa,
\begin{equation}\label{Q_def}
F_\nu^N(z)= z^{-1}e^{-z}Q_{N,\nu}(1/z)
\end{equation}
where the $Q_{N,\nu}$ are polynomials whose degrees satisfy the following bounds for non-negative integers $m$ and $n$,

\begin{align}\label{max_poly_degrees}
&\deg  Q_{2m,2n+1}\leq 2(n+m),\hspace{2mm} \deg Q_{2m+1,0}\leq 2m,\\
& \deg Q_{2m+1,2n}\leq 2(n+m)-1 \text{ for } n>0.
\end{align}
\end{lemma}
\begin{proof}
Starting with the degrees of the base cases from Lemma \ref{F_leading_order}, the recursion relation \req{F_recursion2} implies that the degree increments by at most $2$ every time $N$ is increased by $2$.
\end{proof}
 In the following section we  show that  there are additional cancellations which cause the $Q_{j,k}$'s to have less than the maximum possible degree indicated here.

Before moving on,  note that another family of integrals of the form \req{H_def} for which similar results can be obtained using the recursion relation \req{F_recursion} is
\begin{equation}
H^N_\nu(z)=\int_0^{\pi/2} \cos^N(\theta)J_\nu(z\cos(\theta))d\theta,
\end{equation}
where the base cases are established via the following identity, found in  Watson $5\!\cdot\! 43$ (1) \cite{Watson}.
\begin{lemma}
Let $\Re(\mu+\nu)>-1$.  Then
\begin{equation}
\frac{2}{\pi}\int_0^{\pi/2}J_{\mu+\nu}(2z\cos(\theta))\cos((\mu-\nu)\theta)d\theta=J_\mu(z)J_\nu(z).
\end{equation}
\end{lemma}
This direction will not be pursued further here.

\section{Proof of the Main Theorem.}\label{sec:poly_degree}
A proof of Theorem \ref{main_theorem} will be given in this section.  The main step involves  using the recursion relations derived in the prior section to prove the following bound on polynomial degrees.
\begin{lemma} \label{deg_lemma}
Recall the definition of the polynomials $Q_{N,\nu}$ for $N$ odd and $\nu$ even or vice versa, \req{Q_def}.  Their degrees satisfy
\begin{equation}\label{Q_deg_tight}
\deg Q_{N,\nu}\leq\begin{cases} \nu-1 &\mbox{if } N<\nu \\ 
 N-1& \mbox{if } N>\nu. \end{cases}\\
\end{equation}

For $N$ odd and $\nu$ even or vice versa,
\begin{equation}\label{P_def}
G_{\nu}^N(z)=z^{-2}e^{-z}P_{N,\nu}(1/z)
\end{equation}
 where $P_{N,\nu}$ are polynomials with  
\begin{equation}\label{P_deg_tight}
\deg P_{N,\nu}\leq \deg Q_{N,\nu}-1.
\end{equation}
\end{lemma}
\begin{proof}
The case $\nu=0$ follows from the  bound \req{max_poly_degrees}. $\nu=1$ is the first example where  nontrivial cancellation occurs. Using \req{F_recursion} one finds
\begin{align}
F^0_1(z)=&\frac{\pi}{2}z^{-1}e^{-z},\hspace{2mm}  F^2_1(z)=\frac{\pi}{2}z^{-1}e^{-z}(1+1/z).
\end{align}
Note that the order of the polynomial increases by one, not two as would be naively expected from \req{max_poly_degrees}.  The recursive formula \req{F_recursion} implies that the degree can increment by at most $2$ for each subsequent increase in $N$ by $2$. Explicit computation shows that it does increase by $2$, at least initially.  This proves the lemma for $\nu=1$.

We will assume $\nu>1$ for the remainder of this proof.  Using the recursion \req{F_recursion} gives
\begin{align}
F^{N+2}_{\nu}(z)=&z^{-1}e^{-z}\left[((1-\nu^2)z^{-2}+z^{-1}+1)Q_{N,\nu}(1/z)\right.\notag\\
&\left.+(3z^{-3}+2z^{-2})Q^\prime_{N,\nu}(1/z)+z^{-4}Q^{\prime\prime}_{N,\nu}(1/z)\right]\\
= &z^{-1}e^{-z} Q_{N+2,\nu}.\notag
\end{align}

 Suppose $\deg Q_{N,\nu}\leq \nu-1$ for some $N$.  Note that the bounds \req{max_poly_degrees} imply that this holds initially i.e.
\begin{equation}\label{min_deg_bound}
\deg Q_{N_{\min},\nu}\leq \nu-1
\end{equation}
where $N_{\min}=0$ for $\nu$ odd and $N_{\min}=1$ for $\nu$ even.

Denoting the coefficients of a polynomial $Q$ by $Q^j$ we therefore have
\begin{align}
Q_{N+2,\nu}^{\nu+1}=&(1-\nu^2)Q_{N,\nu}^{\nu-1}+3(Q^\prime_{N,\nu})^{\nu-2}+(Q^{\prime\prime}_{N,\nu})^{\nu-3}\notag\\
=&[(1-\nu^2)+3(\nu-1)+(\nu-1)(\nu-2)]Q_{N,\nu}^{\nu-1}\\
=&0\notag
\end{align}
and
\begin{align}
Q_{N+2,\nu}^{\nu}=&(1-\nu^2)Q_{N,\nu}^{\nu-2}+Q_{N,\nu}^{\nu-1}+3(Q^\prime_{N,\nu})^{\nu-3}+2(Q^\prime_{N,\nu})^{\nu-2}+(Q^{\prime\prime}_{N,\nu})^{\nu-4}\notag\\
=&(2\nu-1)(Q^{\nu-1}_{N,\nu}-Q_{N,\nu}^{\nu-2})
\end{align}
where polynomial coefficients of negative degree are defined to be zero.

Therefore $Q_{N+2,\nu}$ has degree at most $\nu$ and if $Q_{N,\nu}^{\nu-1}=Q_{N,\nu}^{\nu-2}$ then the degree is in fact still bounded by $\nu-1$.  This allows us to conclude the following.
\begin{lemma}
For $\nu>1$ let $L_\nu$ denote the first positive integer that is even if $\nu$ is even and odd if $\nu$ is odd such that  $Q_{L_\nu-1,\nu}^{\nu-1}\neq Q_{L_\nu-1,\nu}^{\nu-2}$. Then 
\begin{align}
\deg Q_{N,\nu}\leq \nu-1 \text{ for } N< L_\nu,\hspace{2mm}   \deg Q_{N,\nu}\leq \nu+N-(L_\nu+1)\text{ for } N>L_\nu.
\end{align}
\end{lemma}
Therefore, if $L_\nu=\nu$ for $\nu\geq 2$ then the proof of \req{Q_deg_tight} will be complete.  This will be shown by induction on $\nu$.

From \req{F10_F12}, $L_2=2$ so the base case is proven.  Suppose $L_\nu=\nu$ for some $\nu\geq 2$. Using \req{F_recursion2} gives
\begin{align}
Q_{N+1,\nu+1}(1/z)=[(\nu+1)z^{-1} +1] Q_{N,\nu}(1/z)+z^{-2}  Q^\prime_{N,\nu}(1/z).
\end{align}
Hence,  for $N<\nu$
\begin{align}
Q_{N+1,\nu+1}^{\nu}=&(1+\nu) Q_{N,\nu}^{\nu-1} +Q_{N,\nu}^\nu+(Q^\prime_{N,\nu})^{\nu-2}\\
=&2\nu Q_{N,\nu}^{\nu-1}\notag
\end{align}
and
\begin{align}
Q_{N+1,\nu+1}^{\nu-1}=&(\nu+1)Q_{N,\nu}^{\nu-2}+ Q_{N,\nu}^{\nu-1}+(\nu-2) Q_{N,\nu}^{\nu-2}\\
=&(2\nu-1)Q_{N,\nu}^{\nu-2}+ Q_{N,\nu}^{\nu-1}.\notag
\end{align}
Therefore
\begin{align}
Q_{N+1,\nu+1}^{\nu}-Q_{N+1,\nu+1}^{\nu-1}=-(Q_{N,\nu}^{\nu-1}-Q_{N,\nu}^{\nu-2}).
\end{align}
This proves $Q^{\nu}_{N+1,\nu+1}=Q_{N+1,\nu+1}^{\nu-1}$ for $N+1\leq\nu$ and  $Q^{\nu}_{\nu,\nu+1}\neq Q_{\nu,\nu+1}^{\nu-1}$.  Note that for $\nu$ even, Lemma \ref{F_leading_order} is needed to show that $Q^{\nu}_{0,\nu+1}=Q_{0,\nu+1}^{\nu-1}$, as the recursion misses this case. This proves $L_\nu=\nu$ for all $\nu\geq 2$ by induction. This concludes the proof of \req{Q_deg_tight}.

Next we evaluate $G_\nu^N(z)=F_\nu^{N+2}(z)-F_\nu^N(z)$.  Using \req{F_formula} yields
\begin{align}
G_\nu^N(z)=z^{-1}e^{-z}&\left(\sum_{s=0}^{M_{N+2,\nu}}\left[\sum_{n+r=s}a_n\Gamma(r+1/2)\sum_{k+m=r}b_{k,N+2} c_{m,n}\right] z^{-s}\right.\\
&\left.-\sum_{s=0}^{M_{N,\nu}}\left[\sum_{n+r=s}a_n\Gamma(r+1/2)\sum_{k+m=r}b_{k,N} c_{m,n}\right] z^{-s}\right).
\end{align}
 The $s=0$ term of the expression in parentheses is
\begin{equation}
a_0c_{0,0}\Gamma(1/2)(b_{0,N+2}-b_{0,N})=0.
\end{equation}
Therefore, another $z^{-1}$ term can be factored out to give
\begin{align}
G_\nu^N(z)=z^{-2}e^{-z}P_{N,\nu}(1/z) 
\end{align}
where $P_{N,\nu}$  is a polynomial with degree at most $M_{N+2,\nu}-1$.  This, together with the uniqueness of asymptotic expansions, proves \req{P_def} and \req{P_deg_tight} and completes the proof this lemma.
\end{proof}

We conclude by showing that Theorem \ref{main_theorem} is a simple consequence of the above lemma, combined with the asymptotic expansions from Corollary \ref{FG_asymp}. 

The uniqueness of  asymptotic expansions in the asymptotic sequence $e^{z}z^{-n}$~\cite{Olver} implies that the coefficients in the polynomials $Q_{N,\nu}$ and $P_{N,\nu}$ must agree with those in the asymptotic expansions \req{F_asymp} and \req{G_asymp} respectively.  The bounds \req{Q_deg_tight} and \req{P_deg_tight} on the polynomial degrees therefore imply Theorem \ref{main_theorem}.

\section*{Acknowledgments}
This work was motivated by the numerical discovery of some of the results presented here. These discoveries were made in the process of studying neutrino freeze-out in the early Universe in collaboration with Johann Rafelski and Cheng Tao Yang, see \cite{Birrell_const}. This work has been supported by US Department of Energy, Office of Science, Office of Nuclear Physics  under  award number DE-FG02-04ER41318 and was conducted with Government support under and awarded by DoD, Air Force Office of Scientific Research, National Defense Science and Engineering Graduate (NDSEG) Fellowship, 32 CFR 168a.

\appendix

\section{Neutrino Scattering}\label{app:neutrino}
In this appendix we outline the connection of the integrals \req{G_def} to neutrino scattering that motivated this paper.  The following is based on results given in \cite{Birrell_const}.

Consider $2$-body to $2$-body reactions between particles in kinetic equilibrium at a common temperature $T$ and with distinct fugacities $\Upsilon_i$, $i=1,2,3,4$. The reaction rate for the reaction $1+2\rightarrow 3+4$ in the Boltzmann limit is given by
\begin{align}
R=&(\Upsilon_1\Upsilon_2-\Upsilon_3\Upsilon_4)/\tau,\\
1/\tau=&\frac{T}{2^{10}\pi^5 }\int_{s_0}^\infty \frac{ rr^{'}}{\sqrt{s}}\left(\int_{-1}^1  \langle |\mathcal{M}|^2\rangle(s,t(y))dy\right) K_1( \sqrt{s}/T)ds,
\end{align}
where $\langle|\mathcal{M}|^2\rangle$ is the squared matrix element of the interaction, and $r,r^\prime,t(y)$ depend on the particulars of the reaction. See \cite{Birrell_dis,Birrell_const} for definitions and further physics background.

 The matrix element for electron-positron annihilation into neutrinos, $e^-+e^+\leftrightarrow\nu_e+\bar{\nu}_e$, is given by
\begin{align}
\langle|\mathcal{M}|^2\rangle=32G^2_F\bigg[&(1+2\eta)^2\frac{(s+t-m^2_e)^2}{4}+(2\eta)^2\frac{(m^2_e-t)^2}{4}+\eta(1+2\eta)m^2_es\bigg],
\end{align}
where $m_e$ is the electron mass, $\eta=\sin(\theta_w)^2$, $\theta_w$ is the Weinberg angle, and $G_F$ is Fermi's constant. For this process
\begin{align}
&r=\sqrt{s-4m_e^2},\hspace{2mm} r^{'}=\sqrt{s},\hspace{2mm} t(y)=\frac{1}{4}(2rr^{'}y-r^2-(r^{'})^2),\\
&\int_{-1}^1  \langle |\mathcal{M}|^2\rangle(s,t(y)) dy=\frac{16G_F^2}{3}s\left[(16\eta^2+8\eta-1)m_e^2+(8\eta^2+4\eta+1)s\right].\notag
\end{align}

Therefore, the rate constant in the  Boltzmann limit corresponding to this process is
\small
\begin{align}
&1/\tau_{e^-e^+\leftrightarrow\nu_e\bar{\nu}_e}\notag\\
=&\frac{T}{2^{10}\pi^5 }\int_{s_0}^\infty \frac{ rr^{'}}{\sqrt{s}}\left(\int_{-1}^1  \langle |\mathcal{M}|^2\rangle(s,t(y)) dy\right) K_1(\beta \sqrt{s})ds\\
=&\frac{G_F^2T}{3(2^{6})\pi^5 }\int_{4m_e^2}^\infty s\sqrt{s-4m_e^2}\left[(16\eta^2+8\eta-1)m_e^2+(8\eta^2+4\eta+1)s\right]K_1(\beta \sqrt{s})ds.\notag
\end{align}
\normalsize
Making the change of variables $z\cosh(t)=\sqrt{s}/T$, $z=2m_e/T$, the integral becomes
\begin{align}\label{ee_nunu_rate}
&1/\tau_{e^-e^+\leftrightarrow\nu_e\bar{\nu}_e}=\frac{ G^2_FT^8}{(2\pi)^5}z^7\int^\infty_1f(\cosh(t))K_1(z\cosh(t))\sinh^2(t)dt,\\
 &f(u)=Au^5+Bu^3,\hspace{2mm} A=\frac{1}{3}(8\eta^2+4\eta+1),\hspace{2mm} B=\frac{1}{3}(4\eta^2+2\eta-1/4).\notag
\end{align}
The integral \req{ee_nunu_rate} is of the form considered in \req{G_def} with both $N$ and $\nu$ odd.  Theorem \ref{main_theorem} does  not apply but Theorem \ref{asymp_approx} can still be used to obtain  the asymptotic expansions as $z\rightarrow\infty$ of the relevant integrals 
\begin{align}
G_1^3(z)=&\frac{\pi}{2}z^{-2}e^{-z}\left(1+\frac{9}{2}z^{-1}+\frac{81}{8}z^{-2}+\frac{165}{16}z^{-3}+O(z^{-4})\right),\\
G_1^5(z)=&\frac{\pi}{2}z^{-2}e^{-z}\left(1+\frac{15}{2}z^{-1}+\frac{285}{8}z^{-2}+\frac{1875}{16}z^{-3}+O(z^{-4})\right).
\end{align}
Using Lemma \ref{Watson_int2} along with the recursion relation \req{F_recursion} also gives the exact formulas
\begin{align}
G^{3}_1(z)=&\frac{1}{4}z^{-1}K_0(z/2)^2+2z^{-2}K_0(z/2)K_1(z/2)+\left(\frac{1}{4}+4z^{-2}\right)z^{-1}K_1(z/2)^2,\\
G^{5}_1(z)=&\left(\frac{1}{4}+6z^{-2}\right)z^{-1}K_0(z/2)^2+\left(\frac{7}{2}+48z^{-2}\right)z^{-2}K_0(z/2)K_1(z/2)\notag\\
&+\left(\frac{1}{4}+10z^{-2}+96z^{-4}\right)z^{-1}K_1(z/2)^2.
\end{align}

\end{document}